\documentclass[12pt]{amsart}

\usepackage{comment}
\usepackage{color}
\usepackage{amsmath}
\usepackage{amsthm}
\usepackage{amssymb}
\usepackage{graphicx}
\usepackage{enumerate}
\usepackage{amsfonts}
\usepackage{palatino}
\usepackage{mathrsfs}
\usepackage{mathdots}

\usepackage{color}

\numberwithin{equation}{section}
\theoremstyle{plain}

\newtheorem{Corollary}[equation]{Corollary}
\newtheorem*{Corollary*}{Corollary}
\newtheorem{Theorem}[equation]{Theorem}
\newtheorem*{Theorem*}{Theorem}
\newtheorem{Lemma}[equation]{Lemma}
\theoremstyle{definition}

\newtheorem{Example}[equation]{Example}
\newtheorem{Remark}[equation]{Remark}

\allowdisplaybreaks

\usepackage{enumitem}
\setlist[enumerate,1]{label=(\alph*),font=\upshape}

\setlist[enumerate,2]{label=(\roman*),font=\upshape}

\def\HH{\mathscr{H}}

\def\Mult{\mathfrak{M}}
\def\P{\mathscr{P}}

\def\M{\mathbf{M}}
\def\B{\mathbf{B}}

\def\C{\mathbb{C}}

\def\D{\mathbb{D}}
\def\T{\mathbb{T}}
\def\N{\mathbb{N}}

\def\phi{\varphi}

\newcommand{\beqa}{\begin{eqnarray*}}
\newcommand{\eeqa}{\end{eqnarray*}}

\newcommand{\kk}{\mathbf{k}}

\renewcommand{\leq}{\leqslant}
\renewcommand{\le}{\leqslant}
\renewcommand{\subset}{\subseteq}

\newcommand{\DD}{\mathcal{D}}

\title[B\'{e}zout's equation]{Sharp estimates of the solutions to B\'{e}zout's polynomial equation and a corona theorem}

\author[Fricain]{Emmanuel Fricain}
 \address{Laboratoire Paul Painlev\'e, Universit\'e de Lille, 59 655 Villeneuve d'Ascq C\'edex }
 \email{emmanuel.fricain@univ-lille.fr}

\author[Hartmann]{Andreas Hartmann}
\address{Univ. Bordeaux, CNRS, Bordeaux INP, IMB, UMR 5251, F-33400, Talence, France}
\email{Andreas.Hartmann@math.u-bordeaux.fr}

\author[Ross]{William T. Ross}
	\address{Department of Mathematics and Statistics, University of Richmond, Richmond, VA 23173, USA}
	\email{wross@richmond.edu}
	
		\author[Timotin]{Dan Timotin}
	\address{Simion Stoilow Institute of Mathematics of the Romanian Academy, PO Box 1-764, Bucharest 014700, Romania}
	\email{Dan.Timotin@imar.ro}

\thanks{The first author was supported by Labex CEMPI (ANR-11-LABX-0007-01) and the project FRONT (ANR-17-CE40 - 0021). The second author was supported by the project REPKA (ANR-18-CE40-0035)}


\begin{document}

\begin{abstract}
In this paper, we obtain estimates for the solutions to the classical B\'{e}zout equation that are analogous to Carleson's solution to the corona theorem for the bounded analytic functions on the open unit disk. As an application,  we extend some results of Luo and  obtain a corona theorem for the multipliers of a class of de Branges--Rovnyak spaces. 
\end{abstract}

\maketitle

\section{Introduction}

A well-known theorem of \'{E}tienne B\'{e}zout (1730-1783) says that if $A, B$ belong to $\C[z]$ and have no common roots, then there are $R, S \in \C[z]$ with 
\begin{equation}\label{degcond}
 \deg R\le \deg B-1 \quad \mbox{and} \quad \deg S\le \deg A-1,
 \end{equation} such that 
\begin{equation}\label{Bezout_C}
A(z) R(z) + B(z) S(z) = 1 \quad \mbox{for all $z \in \C$.}
\end{equation}
Moreover, if $ A$ and  $B$ are not both constant polynomials, then $ R$ and $S $ are uniquely determined by the degree condition in \eqref{degcond} and will be called the \emph{minimal solutions} of~\eqref{Bezout_C}.

The proof of B\'{e}zout's theorem often presented to algebra students uses a version of the classical Euclidean algorithm for polynomials (noting of course that the lack of common roots for $A$ and $B$ implies that the greatest common divisor of $A$ and $B$ is the constant polynomial one). A result of Sylvester (see Theorem \ref{le:the real Bezout} below) gives an explicit formula for the minimal solutions $R$ and $S$ in terms of a certain matrix equation involving the {\em Sylvester resultant}. 

Inspired by Carleson's corona theorem (see the discussion below), one central goal of this paper is to estimate  the coefficients of the polynomials $ S $ and $ R $. To understand what we mean by this, consider the following example. Throughout this paper, we will use 
	\begin{equation}\label{normofaplyiinao}
	\|p\| := \max_{0 \leq j \leq n} |p_j|
	\end{equation}
	for the norm of a polynomial $p(z) = \sum_{j = 0}^{n} p_j z^j$.

\begin{Example}\label{ex:elementary}
		If $n \in \N$ and $0 < \delta < 1$, set 
	$A(z) = z^n$ and $B(z) = z - \delta$. One can work out the corresponding minimal polynomials $R$ and $S$ that satisfy \eqref{degcond} and \eqref{Bezout_C} to be 
	$$R(z) = \frac{1}{\delta^n} \quad \mbox{and} \quad S(z) = - \frac{1}{\delta^n} \sum_{j = 0}^{n - 1} \delta^{n - 1 - j} z^{j}.$$
	Then $\|A\|  = \|B\| = 1$, $|B(0)| = \delta$, and 
	$$\|R\|  = \|S\| = \frac{1}{\delta^n}.$$
  Notice how $n$ is the maximal order of the zeros of $A$, which in this case is a single zero at the origin of order $n$,  while the condition $|B(0)| = \delta$ can be interpreted as a lower bound for $B$ at the zero of $A$.	
	
\end{Example}

Our main theorem proves that the type of phenomenon presented in Example \ref{ex:elementary} always occurs. To state this result,  let $\C[z]$ denote  the vector space of polynomials in the complex variable $z$ with coefficients in $\C$, and for $N \in \N$, let $\P_{N}$ be the finite dimensional subspace
$\P_{N} = \{p \in \C[z]: \deg p \leq N\}$.

\begin{Theorem}\label{th:main estimate}
Let $A \in \P_{N}$ be of the form 
\[
A(z)=\prod_{j=1}^{n}(z-\alpha_j)^{m_j}, 
\]
where $\alpha_1, \alpha_2, \ldots, \alpha_n \in \C$ are distinct and  
$ N=\sum_{j=1}^{n}m_j\geqslant 1$. Fix $ K\in\N $.
There is a $C > 0 $, depending only on $ A$ and $ K $, such that if  $ B \in \P_{K} $ satisfies
	\begin{enumerate} 
		\item $ \| B \|\le 1$, and 
		\item $ |B(\alpha_j)|\geqslant\delta > 0 $ for all $1 \leq j \leq n$,
			\end{enumerate}
	 then the minimal solutions $R$ and $S$ of the  B\'{e}zout equation \eqref{Bezout_C} satisfy 
	 \begin{equation}\label{eq:estimate}
	 (\|R\|^2 + \|S\|^2)^{\frac{1}{2}} \leq \frac{C}{\delta^{\max m_j}}.
	 \end{equation}
	\end{Theorem}

	Theorem \ref{th:main estimate2} below extends this result to multiple polynomials $A$ and $B_1, \ldots, B_L$. 
	Moreover,  as shown in Example \ref{ex:elementary}, the estimate in \eqref{eq:estimate} is sharp in that the exponent $\max\{m_j: 1 \leq j \leq n\}$ on $\delta$ can not be lowered.  It seems  rather surprising to us  that such  estimates are lacking in the literature.

	The statement of Theorem~\ref{th:main estimate} is asymmetric in the polynomials $ A $ and $ B $: $ A $ is fixed (and determines the constant $ C $), while $ B $ is chosen freely in the finite dimensional space $\P_{K}$.  As we will see in Corollary~\ref{co:rough estimate} below, the formula for the solutions of the B\'ezout equation involving the Sylvester determinant, given in Section~\ref{se:sylvester}, can be used to obtain a more symmetric result for the coefficient estimates. However, the  estimate from the symmetric version can be much less precise in terms of the exponent on $ \delta $ than the asymmetric one presented in Theorem~\ref{th:main estimate}.

	
%
 
%
Theorem~\ref{th:main estimate} is connected to a related problem, known as the corona problem \cite{CP} for $H^{\infty}$, the space of  bounded analytic functions on the open unit disk $\D = \{z \in \C: |z| < 1\}$ endowed with the standard norm $\|f\|_{\infty} := \sup_{z \in \D} |f(z)|$. In a way, the corona theorem can be viewed as a B\'{e}zout theorem for $H^{\infty}$.   In 1962, Carleson \cite{MR141789} answered a conjecture proposed by Kakutani \cite{MR5778} in 1941 and proved that if $(f_{j})_{j = 1}^{N}$ is a finite sequence in $H^{\infty}$ which satisfies 
	$$0 < \delta \leq \Big(\sum_{j = 1}^{N} |f_j(z)|^2\Big)^{\frac{1}{2}} \leq 1 \quad \mbox{for all $z \in \D$},$$
	then there is a finite sequence $(g_j)_{j = 1}^{N}$ in $ H^{\infty}$ such that 
	$$\sum_{j = 1}^{N} f_{j}(z) g_{j}(z) = 1 \quad \mbox{for all $z \in \D$}$$
	and
	$$\Big(\sum_{j = 1}^{N} |g_j(z)|^2\Big)^{\frac{1}{2}} \leq C \quad \mbox{for all $z \in \D$}.$$
	In the above, the constant $C$ initially depends  only on $\delta$ and   $N$. Further investigations have shown that  $ C=C(\delta) $ can be chosen to only depend on $ \delta $. Various estimates of $C(\delta)$ were explored 
	in~\cite{hormander, MR629839, Uch}. The latter  paper of Uchiyama contains the sharpest known estimate 
	\begin{equation}\label{uchi}
	C(\delta) \leq C \frac{1}{\delta^2} \log \frac{1}{\delta},
	\end{equation}
	where $C$ is an absolute constant. On the other hand, Treil \cite{MR1945294} proved that one can do no better than 
	$$C \frac{1}{\delta^2} \log \log \frac{1}{\delta}.$$
	
	Carleson's corona theorem has been generalized in many directions. For example, Tolokonnikov \cite{MR629839} and Rosenblum \cite{MR570865} independently proved a version of Carleson's theorem for an {\em infinite} sequence $(f_{j})_{j = 1}^{\infty}$ of $H^{\infty}$ functions. There are also many generalizations of the corona theorem to matrix and operator-valued functions that are important to control theory and similarity problems. For example, Nikolski's book \cite[Ch.9, Sec.~2]{MR1864396} relates the operator-valued corona problem to one-sided invertibility of Toeplitz operators.

	 Another direction of inquiry  stems from the fact that $ H^\infty $ is the multiplier algebra for the Hardy--Hilbert space $ H^2 $ (see Section \ref{H2Hbbasics}) which inspires one to prove corona type theorems for sequences in the multiplier algebra of  other reproducing kernel Hilbert space  of analytic functions \cite{Luothesis, MR4363747, MR2057771}. 
	A particular class of such reproducing kernel Hilbert spaces, which has received quite a lot of attention over the past several decades, arises from  the de Branges--Rovnyak spaces $ \HH(b) $ \cite{Sa, FM1, FM2}. These spaces are an important class of linear submanifolds of the classical Hardy space $H^2$ that appear when modeling certain Hilbert space contractions.  In Theorem \ref{th:corona} we use our B\'{e}zout estimates to obtain a corona theorem for the multiplier algebras of de Branges--Rovnyak spaces $\mathscr{H}(b)$ in the case where $ b $ is any rational  function (but not a finite Blaschke product)   in the closed unit ball of $H^{\infty}$. Our results are related to those from \cite{Luothesis, MR4363747}.

\section{The classical B\'{e}zout theorem}\label{se:sylvester}

\subsection{B\'ezout's formula and the Sylvester matrix}
 
The next theorem gives a useful formula for the solution to the B\'{e}zout problem involving the Sylvester matrix (see \cite[p.~200]{MR1878556} or  \cite[p.~77]{MR2122859}).

\begin{Theorem}\label{le:the real Bezout}
If  
$$A(z) = \sum_{j = 0}^{N} A_{j} z^j \quad \mbox{and} \quad B(z) = \sum_{j = 0}^{K}  B_j z^j$$
are not both constant polynomials, and have no common zeros, then there are unique polynomials 
$$R(z) = \sum_{j = 0}^{K - 1} R_j z^j \quad \mbox{and} \quad S(z) = \sum_{j = 0}^{N - 1} S_j z^j $$ such that 
\begin{equation}\label{eq:first general Bezout}
RA+SB\equiv 1.
\end{equation}
The coefficients $ R_0,\dots, R_{K-1} $ and $ S_0, \dots, S_{N-1} $ are solutions to the system 
\begin{equation}\label{eq:system for coefficients}
\mathfrak{S} \bf x=\bf e,
\end{equation}
where $\mathfrak{S}$ is the $(N + K) \times (K + N)$ matrix 
\bigskip\bigskip
$$\mathfrak{S} := \begin{bmatrix}
A_{0} &  0 & \cdots & 0 & B_{0} & 0 & \cdots & 0\\[5 pt]
A_1 & A_0 & 0 & \cdots & B_1 & B_{0} & \cdots & 0\\
A_2 & A_1 & \ddots & 0 & B_{2} & B_1 & \ddots & 0\\
\vdots & \vdots & \ddots & A_{0} & \vdots & \vdots & \ddots & B_0\\
\vdots  & \vdots & \vdots & A_1 & \vdots & \vdots & \vdots & B_1\\
A_{N} & \vdots &  \ddots & \vdots & B_K & \vdots  & \vdots & \vdots\\
0 & A_{N} & \cdots & \vdots & 0 & B_{K} & \cdots & \vdots\\
0 & 0 & \ddots & \vdots & 0 & 0 & \ddots & \vdots\\
\vdots & \vdots & \ddots & A_{N} & 0 & 0 & 0 & B_{K}
\end{bmatrix},$$
$\hskip 1.3in \underbrace{\quad  \quad \quad \quad \quad  \quad \quad \quad  }_{K} \quad \underbrace{\quad \quad \quad \quad \quad  \quad  \quad}_{N}$
\vskip .10in
$$\mathbf{x} = [R_0, R_1, \ldots, R_{K - 1}, S_{0}, S_1, \ldots, S_{N - 1} ]^{T},$$
and 
$$\mathbf{e} = [1, 0, 0, \ldots, 0]^{T}.$$
Moreover, if $ \alpha_1,\dots, \alpha_N $ are the roots of $ A $ and $ \beta_1,\dots, \beta_K $ are those for $B$, counted with multiplicities, then
\begin{equation}\label{eq:magic formula}
\det{\mathfrak{S}}=A_N^KB_K^N\prod_{i=1}^{N}\prod_{j=1}^{K}(\alpha_i-\beta_j).
\end{equation}
In particular, 
\begin{equation}\label{eq:modulus of S}
|\det{\mathfrak{S}}|=|B_K|^N|A(\beta_1)\cdots A(\beta_K)|
=|A_N|^K|B(\alpha_1)\cdots B(\alpha_N)|.
\end{equation}
\end{Theorem}

\begin{Example}\label{ex:simplest}
Let $\delta > 0$. If 
$$A(z) = z^2  = 0 + 0 z + 1 z^2 \quad \mbox{and} \quad B(z) = z - \delta,$$
then 
$$\mathfrak{S} = \begin{bmatrix}
0 & -\delta & 0 \\
 0 & 1  & -\delta \\
 1 & 0 & 1  
\end{bmatrix}$$
and 
$$\begin{bmatrix}
R_0\\
S_0\\
S_1
\end{bmatrix} = \mathfrak{S}^{-1} \begin{bmatrix} 1\\ 0\\0\end{bmatrix} = \begin{bmatrix} \frac{1}{\delta^2 }\\ -\frac{1}{\delta}\\-\frac{1}{\delta^2}\end{bmatrix}.$$
Thus, 
$$R(z) = \frac{1}{\delta^2} \quad \mbox{and} \quad S(z) = -\frac{1}{\delta} - \frac{1}{\delta^2} z$$ which is the $ n=2 $ case from  Example~\ref{ex:elementary}. 
\end{Example}

\begin{Remark}\label{remarskjdfsdfsdfiii}
\hfill
\begin{enumerate}
\item Requiring that  $ A$ and $B $ are not both constant polynomials ensures the uniqueness of the solutions $ R $ and $ S $ to~\eqref{eq:first general Bezout}. We will tacitly make this nontriviality assumption in the sequel.
\item From the identity
$1=RA+SB=RA+(S/c)(cB),$ it follows that if we  replace $ B $ by $ cB $ ($c\not=0$), then $ S $ gets replaced by $ S/c $, while $ R $ does not change. We will use this rescaling several times.
\item As a consequence of the formulas~\eqref{eq:system for coefficients} and \eqref{eq:magic formula}, along with Cramer's rule, the coefficients of the polynomials $ R$ and $S $ are continuous functions of the roots of $ A,B $, as long as $A$ and $B$ do not have any common roots.

\end{enumerate}
\end{Remark}

Theorem~\ref{le:the real Bezout} yields the following estimate of the coefficients of the polynomials $ R,S $ stated  in a manner similar to Carleson's corona theorem presented in the introduction.

\begin{Corollary}\label{co:rough estimate}
With the notation above, let $ \|A\|, \|B\|\le 1 $, and
\[
 \min\{|A(\beta_j)|, |B(\alpha_i)|, 1 \leq i \leq N, 1 \leq j \leq K\}=\delta>0.
\]
 Then
\begin{equation}\label{eq:rough estimates}
(\|R\|^2 + \|S\|^2)^{\frac{1}{2}} \le \frac{\sqrt{2} (N+K-1)!}{\min\{|A_N|^K\delta^N, |B_K|^N\delta^K  \}}.
\end{equation}
\end{Corollary}

\begin{proof}
	Apply Cramer's rule to the system~\eqref{eq:system for coefficients}. For each coefficient of $ R $ and $ S $, the determinant in the numerator of the formula is a homogeneous function, with $ (N+K-1)! $ terms, in the coefficients of $ A $ and $ B $. By assumption, the latter are all smaller in modulus than~1. 
The denominator  $ \det{\mathfrak{S}} $ can be estimated using~\eqref{eq:modulus of S}.
\end{proof}

Compared to Theorem \ref{th:main estimate}, the estimate given in \eqref{eq:rough estimates}  may be rather rough as seen in the following example.


\begin{Example}\label{ex:new example}
	Fix $0 < \eta < 1$ and define 
	$$ A(z)=z(z-1) \quad \mbox{and} \quad  B(z)=(z-\eta)(z-1+\eta).$$ Then
		\[
	\delta=	\min\{ |A(\eta)|, |A(1-\eta)|, |B(0)|, |B(1)|  \}=\eta(1-\eta).
	\]
	 A computation using the identity $ B(z)=A(z) +\eta(1-\eta)$ leads to
	\[
	R(z)=-\frac{1}{\delta} \quad \mbox{and} \quad S(z)=\frac{1}{\delta}.
	\]
	Hence $(\|R\|^2 + \|S\|^2)^{\frac{1}{2}} =  \sqrt{2}\delta^{-1}$, as expected from Theorem \ref{th:main estimate}, 
	while	 the right hand side of~\eqref{eq:rough estimates} yields 
	$(\|R\|^2 + \|S\|^2)^{\frac{1}{2}} \leq 6 \sqrt{2}  \delta^{-2}$. 
	
\end{Example}

\subsection{B\'ezout formula and interpolation}
Important to our proof of Theorem \ref{th:main estimate} is the following alternate way of finding the solutions $ R$ and $S $ to \eqref{Bezout_C} via interpolation. In standard Lagrange interpolation, one is given distinct complex numbers $ x_1, \ldots, x_t $ (nodes) and complex numbers $y_1 \ldots, y_t$ (targets) and asked to produce a $p \in \P_{t - 1}$ such that $p(x_j) = y_j$ for all $1 \leq j \leq t$.
This unique polynomial $p \in \P_{t - 1}$ is given by the formula
\begin{equation}\label{eq:interpolation difference}
p(x) = y_1 \delta_{1}(x) + y_2 \delta_2(x) + \cdots + y_t \delta_{t}(x), \end{equation}
where
$$ \delta_{j}(x) = \frac{\prod_{i = 1; i \not = j}^{t} (x - x_i)}{\prod_{i = 1; i \not = j}^t (x_j - x_i)}.$$

Hermite interpolation extends Lagrange interpolation to include specifying derivatives. 

\begin{Lemma}[Hermite interpolation]\label{le:hermite interpolation}
	Let $ x_1, \dots, x_t \in \C$ be distinct, $\ell_1, \ell_2, \ldots, \ell_t \in \N$,  and for each $1 \leq j \leq t$, let $ y_j^0, \dots, y_j^{\ell_j-1} \in \C$. 
	
	\begin{enumerate}
		\item If $ L=\ell_1+\dots+\ell_t $, then there is a unique $p \in \P_{ L-1}$, such that 
		$$ p^{(k)}(x_j)=y_j^k  \quad \mbox{for all $0 \le k\le \ell_j-1$ and $1\le j\le t$.}$$ 
		
		\item If  $ x_j $ and $ \ell_j $ are fixed, one can write 
		$$ p=p_1+\dots+p_t,$$
		where for each $1 \leq j \leq t$, the polynomial $p_j$ and its coefficients depend linearly on $(y_{j}^{k})_{k = 0}^{\ell_{j}- 1} \in \C^{\ell_j}$. In fact, 
		$$ p_j^{(k)}(x_j)=y_j^k \quad \mbox{for all $ 0\le k\le \ell_j-1 $,}$$ while, if $ i\not=j $, then 
		$$ p_j^{(k)}(x_i)=0 \quad \mbox{for all $ 0\le k\le \ell_i-1 $}.$$
	\end{enumerate}
\end{Lemma}

Let $ \alpha_1, \alpha_2, \ldots, \alpha_I$ denote the {\em distinct} roots of $A$ with corresponding multiplicities $\mu_1, \mu_2, \ldots, \mu_{I}$ and  $\beta_1, \beta_2, \ldots, \beta_{J}$ denote the {\em distinct} roots of $B$ with corresponding multiplicities $\nu_1, \nu_2, \ldots, \nu_{J}$. 
Below, the roots $\alpha_j$ and $\beta_j$ are no longer counted with multiplicities (different from  the notation from Theorem \ref{le:the real Bezout}). Also recall that 
$$N = \deg A = \sum_{i = 1}^{I} \mu_{j} \quad \mbox{and} \quad K =  \deg B = \sum_{j = 1}^{J} \nu_j.$$

\begin{Corollary}\label{co:bezout by interpolation}
With the notation above, define:
\begin{itemize}
	\item[--] $ R \in \P_{K - 1}$ with the property that, for each $ 1\le j\le J $,
\begin{equation}\label{eq:interpR}
	R(\beta_j)A(\beta_j)=1, \quad
	(RA)^{(k)}(\beta_j)=0\quad\text{ for }1\le k\le\nu_j-1;
\end{equation}

	\item[--] $ S \in \P_{N - 1}$ with the property that, for each $ 1\le i\le I $,
	\begin{equation}\label{eq:interpS}
		S(\alpha_i)B(\alpha_i)=1, \quad
		(SB)^{(\ell)}(\alpha_i)=0\quad\text{ for }1\le\ell\le\mu_i-1.
	\end{equation}
	
\end{itemize}

Then,

\begin{enumerate}
	\item  $ R,S $ are  the minimal solutions of B\'ezout's equation~\eqref{eq:first general Bezout}.
	
	\item One can decompose  $R$ as 
	$$R=\sum_{j = 1}^J R_j,$$  where the coefficients of  the polynomials $ R_j $ are rational functions of $A(\beta_j), \dots, A^{(\nu_j-1)}(\beta_j)  $, with denominator $ (A(\beta_j))^{\nu_j}\not=0 $.

\item Similarly, one can decompose $S$ as  $$ S=\sum_{i=1}^I S_i ,$$ where the coefficients of the polynomials $ S_i $ are rational functions of
	$B(\alpha_i), \dots, B^{(\mu_i-1)}(\alpha_i)$, with denominator $ (B(\alpha_i))^{\mu_i} \not = 0$.
\end{enumerate}

\end{Corollary}

\begin{proof} (a)
	The values of $ RA+SB $, and its derivatives up to $ \mu_j-1 $ at $ \alpha_j $ and up to $ \nu_j-1 $ at $ \beta_j $,  coincide with those of the constant function $ 1 $. From Lemma~\ref{le:hermite interpolation} it follows that ~\eqref{eq:first general Bezout} holds.
	
	(b) 
	By Lemma~\ref{le:hermite interpolation}, one can write $ R=\sum_{j=1}^J R_j $, where the coefficients of the polynomials $ R_j $ are linear functions (with coefficients depending on $ \beta_j $) of the prescribed values at $ \beta_j $ of $ R $ and its first $ \nu_j-1 $ derivatives.
	Using Leibnitz's formula and induction on $ k $, the definition of $ R$ implies that $ R^{(k)}(\beta_j) $ is a rational function of
	$$A(\beta_j), \dots, A^{(k)}(\beta_j) ,$$ with denominator $ (A(\beta_j))^{k+1}  $.  This last value is different from zero, since we are always assuming that $ A $ and $ B $ have no common zeros.
	
	A similar argument yields (c).
	\end{proof}

\section{Proof of Theorem \ref{th:main estimate}}



%

The proof requires some preliminary set up. For $p(z) = \sum_{j = 0}^{K} p_j z^j \in \P_{K}$, we defined the norm 
$$\|p\| = \max_{0 \leq j \leq K} |p_j|$$ in  \eqref{normofaplyiinao} and used it in the statement of Theorem \ref{th:main estimate}. 
We will also need these two norms on $\P_{K}$:
\begin{equation}\label{threenorms}
\|p\|':=\max_{|z|=1} |p(z)| \quad  \mbox{and} \quad \|p\|'':=\max_{|z|=3} |p(z)|.
\end{equation}
Since $\P_{K}$ is a finite dimensional vector space, all norms on $\P_{K}  $ are equivalent. Hence there is a
 $ D>0 $, depending on $K$, such that 
\begin{equation}\label{eq:equivalence of norms}
\frac{1}{\sqrt{D}}\|p\|\le \|p\|' , \|p\|''\le \sqrt{D}\|p\| \quad \mbox{for all $p \in \P_{K}$.}
\end{equation}
As a consequence, if $\|T\|, \|T\|^{'}, \|T\|^{''}$ denote the corresponding operator norms of a linear transformation $T: \P_{K} \to \P_{K}$, where $\P_{K}$ is endowed with the respective norms $\|p\|, \|p\|^{'}$, and $\|p\|^{''}$, then
\begin{equation}\label{eq:equivalence of norms - operators}
\|T\|\le D\|T\|' \quad \mbox{and} \quad \|T\|\le D\|T\|''.
\end{equation}

The statement and proof of Theorem \ref{th:main estimate} contains constants that will depend  on $ A \in \P_{N}$. To define these constants, we proceed as follows. Recall that $\alpha_1, \ldots, \alpha_n$ are the zeros of $A$ with corresponding multiplicities $m_1, \ldots, m_n$.  
Fix positive constants $\rho$ and $c_1, c_2, \ldots, c_n$ such that
\begin{equation}\label{eq:*}
	|z-\alpha_j|\le 2\rho \implies |A(z)|\geqslant c_j|z-\alpha_j|^{m_j}.
\end{equation}
 If $ \mathfrak B_j  = \{z: |z - \alpha_j| \leq \rho\}$, fix  $\eta$ such that 
\begin{equation}\label{eq:**}
	0 < \eta\le\inf\Big\{|A(z)|: z\notin \bigcup_{j = 1}^{n} \mathfrak B_j  \Big\}.
\end{equation}
Since $ A $ is a monic polynomial of degree $ N $, we can choose $ M>0 $ such that  \begin{equation}\label{emmmmmmm}
 |w|\geqslant M  \implies |A(w)|\geqslant \tfrac{1}{2} |w|^N.
 \end{equation}
 We can also assume that $ 0<\rho, \eta, c_j <1 $  for all $1 \leq j \leq n$ and that $ M>1 $. From \eqref{emmmmmmm} it follows that $ M\geqslant |\alpha_i| $ for all $1 \leq i \leq n$.

Since Theorem \ref{th:main estimate} needs to hold for all $B \in \P_{K}$, we assume from now on that $\deg B = k \leq K$. By Remark~\ref{remarskjdfsdfsdfiii}(c), we may suppose that $ B $ has simple roots $ \beta_1, \dots, \beta_k $ and extend the estimate~\eqref{eq:estimate} by continuity if $B$ has roots of higher multiplicity. Thus, we use the notation 
\begin{equation}\label{56tfgcvbnklIJH}
B(z)=B_0+\dots+B_kz^k=B_k(z-\beta_1)\cdots(z-\beta_k), \quad \beta_i \neq \beta_j.
\end{equation}

To summarize, we have fixed  positive numbers $ D $ (depending only on $ K $) and $ \rho, \eta, M, c_j $ (depending only on $ A $). 
This notation will be used below without further comment.
Our  proof of Theorem~\ref{th:main estimate} begins with a few preliminary lemmas. 

\begin{Lemma}\label{le:division by root}
	 Suppose $ p\in \P_{K} $ and $ p(u)=0 $. If ${\displaystyle q(z)=\frac{p(z)}{z-u}}$, then 
	\[
	\left\| q \right\|\le D\|p\|.
	\]
\end{Lemma}

\begin{proof}
For at least one of the norms $ \|\cdot\|' $ or $ \|\cdot\|'' $ from \eqref{threenorms},   depending on the position of $u$ with respect to the circle $ |z|=2 $, multiplication by $(z - u)^{-1} $ on the subspace $\{p \in \P_{N}: p(u) = 0\}$ is a contraction. The conclusion is now a consequence of~\eqref{eq:equivalence of norms - operators}.
\end{proof}

The next result follows from  Theorem~\ref{le:the real Bezout} and Cramer's rule. 

\begin{Lemma}\label{le:Bezout for this case}
Suppose $ B_k=1 $.
\begin{enumerate}
\item The coefficients of $ R $ are quotients $ \Delta_i/\det{ \mathfrak{S}} $, where $ \Delta_i $ is obtained from $ \det{\mathfrak{S}} $ by replacing one of the first $ k $ columns with a column consisting of 0s, except  a $1$ in the first position.

\item $ |\det{ \mathfrak{S}}|=|A(\beta_1)\cdots A(\beta_k)| $.

\item For each $1 \leq i \leq k$,   $ \Delta_i $ is a polynomial, with coefficients depending on $ A $,  of degree at most $ N $ in the coefficients of $ B $. 

\item Alternatively, using Vi\`ete's formulas which relate polynomial coefficients  to  sums of products of its roots, $ \Delta_i $ is a polynomial, again with coefficients depending on $ A$, of separate degree at most $ N $ in each of the roots $ \beta_1, \dots, \beta_k $.

\end{enumerate}
\end{Lemma}

The proof of Theorem \ref{th:main estimate} will depend on the position of the roots $ \beta_1, \dots, \beta_k $ of $B$. We begin with a simple case.

\begin{Lemma}\label{le:w_i large}
There is a $ C > 0$, depending only on $ A $ and $ K $, such that if $ |\beta_i|\geqslant M $ for all $1 \leq i \leq k$, then $ \|R\|\le C $. 
\end{Lemma}

\begin{proof}
Using Remark~\ref{remarskjdfsdfsdfiii}(b), we can assume that $ B_k=1 $.

By Lemma~\ref{le:Bezout for this case}, the coefficients of $ R $ are quotients $ \Delta_i/\det{\mathfrak{S}} $, where $ \Delta_i $ is a polynomial, again with coefficients depending on $ A $ and $ K $, of separate degree at most $ N $ in each of the roots $ \beta_1, \dots, \beta_k $. 
On the other hand, the definition of $M$ from \eqref{emmmmmmm} implies that 
\begin{align*}
|\det{\mathfrak{S}}|  =|A(\beta_1)\cdots A(\beta_k)|
\geqslant \frac{1}{2^K}|\beta_1|^N\dots |\beta_k|^N .
\end{align*}
From here it follows that  $|\Delta_i|/|\det{\mathfrak{S}}|\le C$ for some $ C > 0$ depending only on $ A $ and $ K $.
\end{proof}

The next lemma is central to the proof of Theorem \ref{th:main estimate}.

\begin{Lemma}\label{le:|w_i-w_j|le rho}
Let $B\in \P_K$ with $\|B\| \leq 1$,  $\deg B=k\le K$, and $B$ has $k$ distinct zeros $\beta_1, \beta_2, \ldots, \beta_k$.
If $ |\beta_i-\beta_j|\le \rho $ for all $1 \leq i, j \leq k$, then there is a $ C>0 $, depending only on $ A$ and $ K $, such that 
$$ \|R\|\le \frac{C}{\delta^{\max m_j}}.$$
\end{Lemma}

\begin{proof}
If $ |\beta_i|\geqslant M $ for all $1 \leq i \leq k$, apply Lemma~\ref{le:w_i large} to obtain the conclusion.

 Otherwise, there is some $1 \leq j \leq k$ with $ |\beta_j|<M $, and then the assumption of the lemma implies that $ |\beta_i|\le M+\rho $ for all $1 \leq i \leq k$.
 Vi\`ete's formulas ensure there is some $ M_1>0 $, depending only on $A$ and $K$,  such that $ |B_i/B_k|\le M_1 $ for all $1 \leq i \leq k $. Recall from \eqref{56tfgcvbnklIJH} that $B(z) = B_0 + B_1 z + \cdots + B_{k} z^{k}$ and 
 define $$ \widetilde B:=\frac{B}{B_k}.$$ Observe that the leading coefficient of $ \widetilde{B} $ is 1, while the other coefficients of $\widetilde{B}$ have modulus at most $ M_1 $. So condition (a) in the statement of Theorem~\ref{th:main estimate} is replaced by $ \|\widetilde B\|\le M_1 $, while (b) is satisfied by $ \widetilde{B} $, since the condition $\|B\| \leq 1$ implies that  $ |B_k|\le1 $. On the other hand, we know from Remark~\ref{remarskjdfsdfsdfiii}(b) that $ R $ does not change. Therefore, in the rest of the proof of the lemma, we will assume that $ B_k=1 $.

We will consider two cases:

\begin{enumerate}

\item[a)] Suppose that $\beta_i \not \in  \bigcup_j \mathfrak B_j $ for all $1 \leq i \leq k$. As stated in Lemma~\ref{le:Bezout for this case}, the coefficients of $ R $ are obtained as quotients $ \Delta_i/\det{\mathfrak{S}} $.  Since  $ \Delta_i $ is a polynomial of degree at most $ N $, with coefficients depending on $ A $,  in the coefficients of $ B $, which are all in modulus at most $ M_1 $, there is some constant $ C > 0 $ such that 
$ |\Delta_i|\le C M_1^N.$ On the other hand,~\eqref{eq:**} implies that 
$$|\det{\mathfrak{S}}|= \prod_{i=1}^{k}| A(\beta_i)|\geqslant \eta^{k}  \geqslant \eta^{K}.$$ Therefore,
\[
\|R\|\le \frac{CM_1^N}{\eta^K},
\]
with no dependence on $ \delta $.

\item[b)] If a) is not true, then there exists some $ j $ such that all $ \beta_i $ satisfy $|\beta_i-\alpha_j|\le 2\rho  $. By condition~\eqref{eq:*} above, and our earlier assumption that $0 < c_j < 1$ for all $1 \leq j \leq n$, we conclude that 
\begin{align*}
|\det{\mathfrak{S}}| & =|A(\beta_1)\cdots A(\beta_{k})|
\geqslant  c_j^k\prod_{i=1}^{k}|\beta_i-\alpha_j|^{m_j}\\&=  c_j^k|B(\alpha_j)|^{m_j} \geqslant c_j^k  \delta^{m_j} \geqslant c_{j}^{K} \delta^{m_j}.
\end{align*}
For the numerators $ \Delta_i $, we use the same estimate as in case a) to obtain 
\[
\|R\|\le \frac{C M_1^N}{c_j^K\delta^{m_j}}
\qedhere
\]

\end{enumerate}
\end{proof}

We are finally ready for the proof of Theorem~\ref{th:main estimate}.
The coefficients of the polynomials $ R $ and $ S$ are estimated separately. 
\vskip .10in

{\bf Estimating $ S $}.
   From Corollary~\ref{co:bezout by interpolation},
  one can decompose $ S $ as $ S=\sum_{i = 1}^{n} S_i $, where 
  the coefficients of $ S_i $ are rational functions in $ B $ and its derivatives up to $ m_i-1 $ evaluated at $ \alpha_i $, and the denominator of $ S_i $ is $ B(\alpha_i)^{m_i} $. Since $ |B(\alpha_i)|\geqslant\delta $, it follows that  $ |B(\alpha_i)|^{m_i}\geqslant \delta^{m_i} $.
  
  Since for all $0 \leq \ell \leq m_{i} - 1$ and $0 \leq i \leq n$, the linear functionals $q \mapsto q^{(\ell)}(\alpha_i)$ are bounded on $\P_{K}$,  the terms appearing in the numerator can be bounded by a fixed polynomial in the  norm of $ B $, and the latter is at most 1. Thus, $ S $ satisfies the estimate in~\eqref{eq:estimate}.

\vskip .10in

{\bf Estimating $ R $}.  
Here the situation becomes more complicated. 
We will   
  prove the required estimate by induction on $ k=\deg B $. The constant $ C $ will change at each step, but this is not an issue since we only require at most $k \leq  K $ induction steps.

If $ k=0 $, then   $ B=c $ is a constant, and applying condition (ii) at any root of $ A $ (which has degree at least 1) yields that $ |c|\geqslant \delta $.  Then~\eqref{eq:estimate} is satisfied with $R \equiv 0$, $S \equiv 1/c$, and $C = 1$. In fact, in this case, we may replace $ \delta^{\max m_j} $ by $ \delta $.

If $k = 1$, we can apply Lemma \ref{le:|w_i-w_j|le rho} since the condition on the $\beta_j$ is automatically satisfied.
We now pass to  the induction step from $ k-1 $ to $ k $, where $k \geqslant 2$. We will now take some extra care in denoting the $k$ simple roots $\beta_1, \beta_2, \ldots, \beta_k$ of $B$. First, we can assume that $\beta_1$ is the smallest in absolute value; then $|\beta_1| < M$, since otherwise we apply Lemma \ref{le:w_i large} to complete the proof. Second, we can assume there exists a root, that we will call $\beta_k$, such that $|\beta_1 - \beta_k| > \rho/2$. Otherwise, for any $1 \leq i, j \leq k$, 
$$|\beta_i - \beta_j| \leq |\beta_i - \beta_1| + |\beta_1 - \beta_j| \leq \frac{\rho}{2} + \frac{\rho}{2} = \rho$$
in which case we can apply Lemma \ref{le:|w_i-w_j|le rho} to complete the proof. The remaining roots will, of course, be labeled $\beta_2, \ldots, \beta_{k - 1}$. 
%

 Since the roots of $ B $ are simple, Corollary~\ref{co:bezout by interpolation} says that $ R $ is  the unique polynomial of degree at most $ k-1 $ that satisfies $ R(\beta_i)=1/A(\beta_i) $ for $ i=1, \dots, k $.
Applying \eqref{eq:interpolation difference}, a direct calculation shows that 
\begin{equation}\label{eq:int dif applied}
R(z)=\frac{z-\beta_{k}}{\beta_1-\beta_{k}}r(z)- \frac{z-\beta_{1}}{\beta_1-\beta_{k}}s(z),
\end{equation}
where $ r $ and $ s $ are the interpolation polynomials corresponding to the targets $ 1/A(\beta_i) $ and the nodes $ \beta_1, \dots, \beta_{k-1} $ and $ \beta_2, \dots, \beta_k $ respectively. Since $r$ and $s$ have degree $k - 2$, we will apply the induction hypothesis to these polynomials. 
To obtain a bound for the norm of $R $, it is enough to estimate the norm of the two terms in the right hand side of \eqref{eq:int dif applied}. We will do this in detail only for the first term since the bound on the second follows in a similar way.

We separate our discussion into two cases. 

\noindent (a): $ |\beta_k|\le 2M $: Let $$ \widetilde B:= \frac{B}{D(z-\beta_k)} $$ and note that Lemma~\ref{le:division by root} implies that  $\| \widetilde B\|\le \|B\|\le 1 $. On the other hand, since $ |\beta_k-\alpha_i|\le 2M+M=3M $ for all $ i $, 
$$|\widetilde B(\alpha_i)|\geqslant \frac{\delta}{3DM} .$$

Since $r $ is an interpolation polynomial corresponding to the nodes $ \beta_1, \dots, \beta_{k-1} $ and the targets 
$$\frac{1}{A(\beta_1)}, \frac{1}{A(\beta_2)}, \ldots, \frac{1}{A(\beta_{k - 1})},$$ Corollary~\ref{co:bezout by interpolation} implies that $ r$ is the polynomial of degree at most $ k-2 $ that satisfies the B\'ezout equation 
$$ A r+\widetilde B\widetilde S \equiv 1$$ for some $ \widetilde S $ with $ \deg \widetilde S<\deg A $.

Now apply the induction hypothesis to obtain the estimate  \[ \|r\|\le \frac{C}{\left(\frac{\delta}{3DM}\right)^{\max m_j}}=\frac{C(3DM))^{\max m_j}}{\delta^{\max m_j}} .\]
When $ |z|=1 $ we have 
$$ \frac{|z-\beta_k|}{|\beta_1-\beta_k|}\le \frac{(2M+1) 2}{\rho} .$$ Therefore, applying~\eqref{eq:equivalence of norms - operators},
\begin{align*}
\left\|\frac{z-\beta_{k}}{\beta_1-\beta_{k}}r(z)\right\|\ & \le 2 D\frac{2M+1}{\rho}\|r\|\\
& \le 2 \frac{CD(2M+1)(3DM))^{\max m_j}}{\rho\delta^{\max m_j}}. 
\end{align*}
This ends the proof of case (a).

\noindent (b): $ |\beta_k|>2M $.
 Instead of $ \widetilde B $,  we now consider the polynomial 
$$ \widehat B:=\frac{\beta_k B}{2D(z-\beta_k)}  ,$$
and we view  the polynomial $ r $ as satisfying the B\'ezout equation
\[
A r+\widehat B \widehat S \equiv1
\]
for some polynomial $\widehat S $  with $ \deg \widehat S<\deg A $.
For any $ |z|=1 $ we have, remembering that $ M>1 $ and writing 
$$
\frac{\beta_k}{z - \beta_k} = \frac{z}{z - \beta_k} - 1,
$$
that 
\begin{equation}\label{eq:second case}
  \left|\frac{\beta_k}{z-\beta_k}\right|\le \frac{1}{2 M - 1} + 1 =  \frac{2M}{2M-1}\le2.
\end{equation}
Therefore, multiplication by 
$$\frac{\beta_k}{2(z-\beta_k)} $$ is a contraction operator in the norm $ \|\cdot\|' $, and from~\eqref{eq:equivalence of norms - operators} it follows that $ \|\widehat B\|\le 1 $.   On the other hand, since $ |\alpha_i|\le M $ and $ |\beta_k|>2M $, we have 
\[
\left| \frac{\beta_k}{\alpha_i-\beta_k} \right|\geqslant 
\frac{|\beta_k|}{|\beta_k|+M}\geqslant \frac{|\beta_k|}{\frac{3}{2} |\beta_k|} =\frac{2}{3},
\]
whence
\[
|\widehat B(\alpha_i)|=\frac{|\beta_k| |B(\alpha_i)|}{2D|\alpha_i-\beta_k|} 
\geqslant 
\frac{1}{3D}|B(\alpha_i)|\geqslant \frac{\delta}{3D} .
\]

We now apply  the induction hypothesis to obtain the bound
\[ \|r\|\le \frac{C(3D )^{\max m_j}}{ \delta^{\max m_j}}  .\]
Moreover, under the assumptions of this case, since $ |\beta_1|<M $ and $ |\beta_k|\geqslant 2M $, we see that 
$$|\beta_1 - \beta_k|  \geqslant |\beta_k| - M \geqslant \frac{|\beta_k|}{2}$$ and thus 
$$ \frac{1}{|\beta_1-\beta_k|}\leqslant \frac{1}{M} \quad \mbox{and} \quad 
\left| \frac{\beta_k}{\beta_1-\beta_k}  \right|\le 2.$$
It now follows that 
\[
\left| \frac{z-\beta_{k}}{\beta_1-\beta_{k}}  \right|\le \frac{1}{M}+2=\frac{2M+1}{M} \quad \mbox{for all $|z| = 1$}.
\]
 Therefore, applying~\eqref{eq:equivalence of norms - operators} yields 
\begin{align*}
\left\|\frac{z-\beta_{k}}{\beta_1-\beta_{k}}r(z)\right\|& \le D\frac{2M+1}{M}\|r\| \\
&\le  
\frac{CD(2M+1)(3D)^{\max m_j+1}}{M \delta^{\max m_j}}.
\end{align*}
This ends the proof of case (b), concluding the induction argument needed to estimate the coefficients of $ R $. The proof of the theorem is now complete.\qed

	We remind the reader that Example~\ref{ex:elementary}  shows that the exponent in $\delta$ from the estimate \eqref{eq:estimate} is best possible.

\section{Extension to several polynomials}

Extending Theorem~\ref{th:main estimate} to several polynomials involves the \emph{plank theorem} \cite{MR46672}.

\begin{Lemma}\label{le:geometry}
For vectors $\mathbf{v}_1,\mathbf{v}_2,\dots,\mathbf{v}_n$ in a Hilbert space $\mathcal{H}$ that satisfy
$$\|\mathbf{v}_i\|_{\mathcal{H}} \geqslant 1 \quad \mbox{for all $1 \leq i \leq n$,}$$ there exists a unit vector $\mathbf{y}\in \mathcal{H}$ such that 
$$|\langle  \mathbf{v}_i, \mathbf{y}\rangle_{\mathcal{H}} |\geqslant \frac{1}{\sqrt{n}} \quad \mbox{for all $1\leq i\leq n$.}$$
\end{Lemma}


Here is our extension of Theorem \ref{th:main estimate} to several polynomials.

\begin{Theorem}\label{th:main estimate2}
Let $A \in \P_{N}$ be of the form 
\[
A(z)=\prod_{j=1}^{n}(z-\alpha_j)^{m_k}, 
\]
where $\alpha_1, \alpha_2, \ldots, \alpha_n$ are distinct and 
$N=\sum_{j=1}^{n}m_j\geqslant 1$. Fix $K \in \N$. 
	Then there is a $ C > 0 $, depending only on $ A $ and $K$, such that, if  $ B_1, \dots , B_L\in \P_{K} $ satisfy the conditions
\begin{enumerate}
 \item ${\displaystyle \sum_{j=1}^{L}\|B_j\|^2\le 1}$ and 
 \item ${\displaystyle \sum_{j=1}^{L}|B_j(\alpha_i)|^2\geqslant \delta^2 > 0}$ for all $1 \leq i \leq n$,
 \end{enumerate}
then
there is an $ R\in \P_{K-1}$, and $S_1, \dots, S_L\in \P_{N-1} $, such that
\begin{equation}\label{s7yyyyYuuhhVHJJhh}
RA+S_1B_1+\dots+S_LB_L \equiv 1,
\end{equation}
and 
 \begin{equation}\label{eq:estimate2}
\Big(\|R\|^2 + \sum_{j = 1}^{L} \|S_j\|^2\Big)^{\frac{1}{2}} \le \frac{C}{\delta^{\max m_i}}.
\end{equation}

	\end{Theorem}

\begin{proof}
Consider the vectors $ \mathbf{v}_1, \ldots, \mathbf{v}_n \in  \C^L $ defined by 
\[
\mathbf{v}_i=\frac{1}{\delta}(B_1(\alpha_i), B_2(\alpha_i), \dots, B_L(\alpha_i))
\]
and note that 
$ \|\mathbf{v}_i\|_{\C^{L}}\geqslant 1 $ for all $1 \leq i \leq n$. Lemma~\ref{le:geometry} produces a  unit vector $\mathbf{y}=(y_1, \dots, y_L)  \in \C^L$ such that
\[
\frac{1}{\sqrt{n}}\le |\langle \mathbf{v}_i,\mathbf{y} \rangle_{\C^{L}}|=\frac{1}{\delta}
\Big|\sum_{j=1}^L  B_j(\alpha_i)\overline y_j\Big| \quad \mbox{for all $1 \leq i \leq n$.}
\]
If $B(z)=\sum_{j=1}^L\overline y_jB_j(z)$, it follows that 
\[
|B(\alpha_i)|\geqslant \frac{\delta}{\sqrt{n}}\geqslant \frac{\delta}{\sqrt{N}}.
\]
Furthermore, 
\[
\|B\|\le \Big( \sum_{j = 1}^{L}|y_j|^2 \Big)^{\frac{1}{2}} \Big( \sum_{j = 1}^{L}\|B_j\|^2 \Big)^{\frac{1}{2}}\le 1 .
\]
Theorem~\ref{th:main estimate} produces an  $ R\in \P_{K-1}$ and an $S\in \P_{N-1} $ that satisfy the conditions 
\begin{equation}\label{eq:estim several poly}
RA+SB\equiv1 \quad \mbox{and} \quad	(\|R\|^2 +  \|S\|^2)^{\frac{1}{2}} \le \frac{C\sqrt{N}^{\max m_i}}{\delta^{\max m_i}}.
\end{equation}

The B\'ezout identity$ RA+SB\equiv1 $  can be written as
\[
RA+\sum_{j=1}^L\bar y_j S B_j \equiv 1.
\]
Moreover,
if we define $ S_j=\bar y_j S $, 
then \eqref{s7yyyyYuuhhVHJJhh} holds along with 
\[
\|R\|^2 + \sum_{j=1}^L\| S_j \|^2= 
\|R\|^2 + \sum_{j=1}^L |y_j|^2 \|S\|^2=\|R\|^2 + \|S\|^2.
\]
The estimate in \eqref{eq:estimate2} now follows from the estimate in \eqref{eq:estim several poly}.
\end{proof}

\begin{Remark}\label{re:uniqueness}
	In Theorem~\ref{th:main estimate} the polynomials $ R,S $ are uniquely determined by the degree condition from \eqref{degcond}. However, Theorem~\ref{th:main estimate2}  only yields  the existence of {\em some} polynomials $ R $ and  $ S_1, \ldots, S_L$ that satisfy the desired estimates.
\end{Remark}

\section{de Branges--Rovnyak spaces}\label{DR-section}\label{H2Hbbasics}

The second main theorem of this paper (Theorem \ref{th:corona}) extends results from  \cite{MR4363747}  and establishes a corona theorem for the multipliers of certain de Branges--Rovnyak spaces. In fact, this corona theorem is what originally drew us to  investigate coefficient estimates of B\'{e}zout's identity. In this section we present some of the basics of de Branges--Rovnyak spaces \cite{FM1, FM2, Sa}, along with  some additional results
which seem to be interesting
on their own. The next section will contain our corona theorem. 

Denote
$$\operatorname{ball}(H^{\infty}) := \Big\{b \in H^{\infty}: \|b\|_{\infty} = \sup_{z \in \D} |b(z)| \leq 1\Big\}.$$
For $b \in \operatorname{ball}(H^{\infty})$, the \emph{de Branges--Rovnyak space} $\HH(b)$ is the reproducing kernel Hilbert space associated with the positive definite kernel 
\begin{equation}\label{eq:original kernel H(b)}
k^{b}_{\lambda}(z) = \frac{1 - b(z) \overline{b(\lambda)}}{1 - \overline{\lambda} z}, \quad \lambda, z \in \D.
\end{equation}

It is known that $\HH(b)$ is contractively contained in the well-studied Hardy space $H^2$ of analytic functions $f$ on $ \D $ for which 
$$\|f\|_{H^2} := \Big(\sup_{0 < r < 1} \int_{\T} |f(r \xi)|^2 dm(\xi)\Big)^{\frac{1}{2}}<\infty,$$
where $m$ is normalized Lebesgue measure on the unit circle $\T = \{\xi \in \C: |\xi| = 1\}$ \cite{Duren, garnett}. For $f \in H^2$, the radial limit 
$\lim_{r \to 1^{-}} f(r \xi) =: f(\xi)$ exists for $m$-almost every $\xi \in \T$ and
\begin{equation}\label{knnhHarsysu}
\|f\|_{H^2}  = \Big(\int_{\T} |f(\xi)|^2 dm(\xi)\Big)^{\frac{1}{2}}.
\end{equation}
Furthermore, Parseval's theorem says that if $f(z) = \sum_{k = 0}^{\infty} a_k z^k$ belongs to $H^2$, then 
\begin{equation}\label{sumofsq}
\|f\|^{2}_{H^2}  = \sum_{k = 0}^{\infty} |a_k|^2.
\end{equation}
Though $\HH(b)$ is contractively contained in $H^2$, it is generally not closed in the $H^2$ norm. In fact, for the $b$ explored in this section, $\HH(b)$ is dense in $H^2$.

Throughout this section, we will assume that $ b \in \operatorname{ball}(H^{\infty}) $ is a rational function that is not a finite Blaschke product. We exclude  the finite Blaschke products from our discussion since we will  be exploring a corona theorem for the multiplier algebra of $\HH(b)$. When $b$ is a finite Blaschke product, $\HH(b)$ becomes the usual model space $H^{2} \ominus b H^2$ and in that case it is well-known that the multiplier algebra of $\HH(b)$ is formed just by the constant functions \cite{MR3526203}. Thus, when $b$ is a finite Blaschke product, any corona theorem concerning the multipliers of $\HH(b)$ becomes a triviality.

Although, for a general $b \in \operatorname{ball}(H^{\infty})$ the contents of $\HH(b)$ seem mysterious, when $b \in \operatorname{ball}(H^{\infty})$ is a rational function (and not a finite Blaschke product)  the description of $\HH(b)$ is quite explicit. 
For such a $b$ there exists a unique  nonconstant rational function $ a $ with no zeros on $\D$ such that $a(0) > 0$ and $ |a(\xi)|^2+|b(\xi)|^2=1 $ for  all $ |\xi|=1$. This function $a$ is called the \emph{Pythagorean mate} of $ b $. In fact, one can obtain $a$ from the Fej\'{e}r--Riesz theorem  (see \cite{MR3503356}). Let $ \xi_1, \dots, \xi_n $ denote the {\em distinct} roots of $ a $ on $ \T $, with corresponding  multiplicities $ m_1, \dots, m_n $, and define the polynomial $ a_1 $ by 
\begin{equation}\label{eq:definition of a}
a_1(z): =\prod_{j=1}^n (z-\xi_j)^{m_j}.
\end{equation}
Results from  \cite{MR3110499, MR3503356} show that $\HH(b)$ has an explicit description as 
\begin{equation}\label{eq:formula for H(b)}
\HH(b)=a_1H^2 \dotplus \P_{N-1},
\end{equation}
where $ N=m_1+\dots+m_n $ and $\dotplus$ above denotes the algebraic direct sum in that $a_1 H^2 \cap \P_{N - 1} = \{0\}$. Moreover, if $ f\in\HH(b) $ is decomposed with respect to \eqref{eq:formula for H(b)} as 
\begin{equation}\label{uUUiipPPS}
 f=a_1\widetilde{f}+p, \quad  \mbox{where $\widetilde{f}  \in H^2$ and  $p \in \P_{N - 1}$},
 \end{equation}
an equivalent norm on $ \HH(b) $  (to the natural one induced by the positive definite kernel $k_{\lambda}^{b}(z)$ above)  is
\begin{equation}\label{eq:norm in h(b)}
\|a_1\widetilde f+p\|^{2}_{b}:=\|\widetilde{f}\|^2_{H^2}+\|p\|^2_{H^2}.
\end{equation}
 It is important to note that $ \|\cdot\|_b $ is only equivalent to the original norm corresponding to the kernel in~\eqref{eq:original kernel H(b)}, and its scalar product as well as the reproducing kernels and the adjoints of operators defined on $\HH(b)$ will be different. 
 With the norm $\|\cdot\|_{b}$ and the corresponding inner product in mind, we need
  to introduce a new notation for the associated reproducing kernels, different from~\eqref{eq:original kernel H(b)}, namely $ \kk^b_\lambda $ (note the bold face). By the term reproducing kernel we mean that $\kk^{b}_{\lambda} \in \HH(b)$ for all $\lambda \in \D$ and 
 \[
 \langle f, \kk^b_\lambda\rangle_b = f(\lambda) \quad \mbox{for all $ f\in\HH(b) $ and $ \lambda\in\D $.}
 \]
 
 Using \eqref{uUUiipPPS} and the standard estimate that any $g \in H^2$ satisfies 
 \begin{equation}\label{gggbigggooooo}
 |g(z)| \leq \frac{\|g\|^{2}_{H^2}}{1 - |z|^2} \quad \mbox{for all $z \in \D$},
 \end{equation}
 we see that for fixed $1 \leq k \leq n$ and for each $f \in \HH(b)$ we have 
 \begin{equation}\label{q}
 f(\xi_k) = \lim_{r \to 1^{-}} f(r \xi_k) = p(\xi_k),
 \end{equation}
 where $f = a_1 \widetilde{f} + p$ with $\widetilde{f} \in H^2$ and $p \in \P_{N - 1}$. 
  In the spirit of \eqref{gggbigggooooo}, the next lemma (interesting in its own right and useful later) yields more
precise information on the boundary behavior of $\HH(b)$ functions.
In particular, it shows that $\HH(b)$ functions admit
tangential limits in suitable approach regions  at each point $\xi_k$ (see Remark \ref{tangentialapproachregionsj} below). 
 
 \begin{Lemma}\label{8767ytyuhjBHHGBGHUHU}
 For each fixed $1 \leq k \leq n$, there is a $c_k > 0$, depending only on $b$, such that for each $f \in \HH(b), \eta > 0$, and $z \in \D$, we have 
 $$|f(z)|^2 \leq (1 + \eta) |f(\xi_k)|^2 + c_k \Big(1 + \frac{1}{\eta}\Big) \frac{|z - \xi_k|^2}{1 - |z|^2} \|f\|_{b}^{2}.$$ 
 \end{Lemma}

\begin{proof}
For fixed $1 \leq k \leq n$, remembering from \eqref{eq:definition of a} that $\xi_k$ is a root of $a_1$, we define the polynomial $a^{\#}(z)$
by
$$a_{k}^{\#}(z) := \frac{a_{1}(z)}{z - \xi_k}.$$
Write $f \in \HH(b)$ as $f = a_1  \widetilde{f} + p$ as in \eqref{uUUiipPPS}. By \eqref{q} we have $f(\xi_k) = p(\xi_k)$, and so  
\begin{align*}
f(z) & = (z - \xi_k)  a_{k}^{\#}(z) \widetilde{f}(z)  + (p(z) - p(\xi_k)) + p(\xi_k)\\
& = (z - \xi_k) \Big(a_{k}^{\#}(z) \widetilde{f}(z) + \frac{p(z) - p(\xi_k)}{z - \xi_k}\Big) + p(\xi_k)\\
& = (z - \xi_k) f_{k}(z) + f(\xi_k),
\end{align*}
where 
\begin{equation}\label{dkjfgjj777yy781144}
f_k(z) = a_{k}^{\#}(z) \widetilde{f}(z)  + \frac{p(z) - p(\xi)}{z - \xi_k} \in H^2.
\end{equation}

Given $\eta > 0$, for any $a,  b > 0$ we have 
$$2 a b \leq \eta a^2 + \frac{1}{\eta} b^2$$ and so
\begin{align*}
|f(z)|^2 & \leq |(z - \xi_k) f_{k}(z) + f(\xi_k)|^2\\
& \leq |f(\xi_k)|^2 + |z - \xi_k|^2 |f_{k}(z)|^2 + 2 |(z - \xi_k) f_{k}(z)| |f(\xi_k)|\\
& \leq (1 + \eta) |f(\xi_k)|^2 + \Big(1 + \frac{1}{\eta}\Big) |z - \xi_k|^2 |f_{k}(z)|^2\\
&\leq (1 + \eta) |f(\xi_k)|^2 + \Big(1 + \frac{1}{\eta}\Big) \frac{|z - \xi_k|^2}{1 - |z|^2} \|f_{k}\|^{2}_{H^2}.
\end{align*}
In the last inequality above, note the use of  \eqref{gggbigggooooo}.

To finish the proof, it suffices to show there exists a $c_k > 0$, depending only on $b$ and $k$, such that 
\begin{equation}\label{qq}
\|f_k\|^{2}_{H^2} \leq c_{k} \|f\|_{b}^{2}.
\end{equation}
The definition of $f_k$ from \eqref{dkjfgjj777yy781144} says that 
\begin{align*}
\|f_{k}\|_{H^2}^{2} & \leq 2 \Big(\|a_{k}^{\#} \widetilde{f}\|_{H^2}^{2} + \Big\|\frac{p - p(\xi_k)}{z - \xi_k} \Big\|_{H^2}^{2}\Big)\\
& \leq 2 \Big(\|a_{k}^{\#}\|_{\infty} \| \widetilde{f}\|_{H^2}^{2} + \Big\|\frac{p - p(\xi_k)}{z - \xi_k} \Big\|_{H^2}^{2}\Big).
\end{align*}
Since the map 
$$p(z) \mapsto \frac{p(z) - p(\xi_k)}{z - \xi_k}$$ is a linear transformation from $\P_{N - 1}$ to itself and 
$\P_{N - 1}$ is a finite dimensional space (and hence all norms on $\P_{N - 1}$ are equivalent), this map is continuous and hence there is a constant $\widetilde{c}_k > 0$ such that 
$$\Big\|\frac{p - p(\xi_k)}{z - \xi_k}\Big\|^{2}_{H^2} \leq \widetilde{c}_{k} \|p\|^{2}_{H^2} \quad \mbox{for all $p \in \P_{N - 1}$}.$$
Thus, 
\begin{align*}
\|f_k\|^{2}_{H^2} & \leq 2 ( \|a_{k}^{\#}\|_{\infty}^{2} \|\widetilde{f}\|^{2}_{H^2} + \widetilde{c}_{k} \|p\|^{2}_{H^2})\\
& \leq  c_{k} (\|\widetilde{f}\|^{2}_{H^2} + \|p\|_{H^2}^{2})\\
& = c_{k} \|f\|_{b}^{2},
\end{align*}
where $c_{k}  = 2 \max (\|a_{k}^{\#}\|_{\infty}^2, \widetilde{c}_{k})$. This verifies \eqref{qq} and thus completes the proof. 
\end{proof}

\begin{Remark}\label{tangentialapproachregionsj}
In the spirit of the above proof,
one can write 
$$f(z)=(z-\xi_k)^{m_k}a_k^\dagger(z)\widetilde{f}(z)+p(z),$$ where
$$a_k^{\dagger}(z)=\prod_{j\neq k}(z-\xi_j)^{m_j},$$
to prove that each $f \in \HH(b)$ admits a boundary
limit at $\xi_k$ in the approach regions
$$\Big\{z \in \D: \frac{|z-\xi_k|^{2m_k}}{1-|z|} \leq c\Big\}, \quad c > 1,$$
which are larger than the standard nontangential (Stolz) regions 
$$\Big\{z \in \D: \frac{|z-\xi_k|}{1-|z|} \leq c\Big\}, \quad c > 1.$$
\end{Remark}


Let 
$$\Mult(\HH(b)) := \{\phi \in \HH(b): \phi \HH(b) \subset \HH(b)\}$$
denote the \emph{multiplier algebra} of $\HH(b)$. Standard results for multiplier algebras, true for any reproducing kernel Hilbert space of analytic functions, say that if $\phi \in \Mult(\HH(b))$, then $ \phi\in H^\infty $, and the multiplication operator 
$M_{\phi} f = \phi f$ is bounded on $\HH(b)$ and satisfies 
\begin{equation}\label{7ygGHYGTG6}
M_{\phi}^{*} \kk^{b}_{\lambda} = \overline{\phi(\lambda)} \kk^{b}_{\lambda} \quad  \mbox{for all $\lambda \in \D.$}
\end{equation}

For general $\HH(b)$ spaces, the multiplier algebra $ \Mult(\HH(b)) $ lacks a complete description \cite{MR1098860, MR1254125, MR1614726}.   In  our case, where $b \in \operatorname{ball}(H^{\infty})$ is rational and not a finite Blaschke product, things again become much simpler. Indeed, \cite[Proposition 3.1]{MR3967886} says that 
\begin{equation}\label{mukltiewrlierrrrrs}
\Mult(\HH(b)) = \HH(b) \cap H^{\infty},
\end{equation}
and \eqref{eq:formula for H(b)} implies that $\phi \in \Mult(\HH(b))$ if and only if 
$$\phi = a_1 \widetilde{\phi} + r,\quad \mbox{where $\widetilde{\phi} \in H^2$, $r \in \mathscr{P}_{N - 1}$,  and $a_1 \widetilde{\phi} \in H^{\infty}.$}$$


In particular, it follows easily from~\eqref{uUUiipPPS} that every polynomial is a multiplier of $ \HH(b) $ (this is also a consequence of more general facts from \cite[Ch. IV]{Sa}). Since all norms on the finite dimensional space $ \P_{N-1}$ are equivalent, we fix a $ C_1  > 0$ that  satisfies
\begin{equation}\label{eq:definition of C_1}
\max\{\|p\|, 	\|p\|_{\Mult(\HH(b))}\}\le C_1\|p\|_b = C_1 \|p\|_{H^2}, \quad p \in \P_{N - 1}.
\end{equation}

 \section{A corona theorem for de Branges--Rovnyak  spaces}

Our corona theorem for $\Mult(\HH(b))$ will be stated in terms of column multipliers. 
For a sequence $\Phi =  (\phi_{j})_{j \geqslant 1}$ of functions in $\Mult(\HH(b))$, define  the column multiplier
\begin{equation}\label{eq:big multiplier}
\M_\Phi: \HH(b)\to \bigoplus_{j=1}^\infty \HH(b), \quad \M_{\Phi}  f =  (\phi_j f)_{j\geqslant 1}.
\end{equation}
When $\M_{\Phi}$ is bounded, its adjoint is given by  
$ \M_\Phi^*=(M_{\phi_1}^*,M_{\phi_2}^*, \dots) .$
As is standard, 
$$\bigoplus_{j = 1}^{\infty} \HH(b) : = \Big\{(f_j)_{j \geqslant 1}:  f_j \in \HH(b), \sum_{j = 1}^{\infty} \|f_{j}\|_{b}^{2} < \infty\Big\}$$
with
$$\|(f_j)_{j \geqslant 1}\|_{\bigoplus_{j=1}^\infty \HH(b)} := \Bigg(\sum_{j = 1}^{\infty} \|f_j\|_{b}^{2}\Bigg)^{\frac{1}{2}}$$
(recall the norm $\|\cdot\|_{b}$ on $\HH(b)$ from \eqref{eq:norm in h(b)}).   
The next lemma generalizes~\cite[Lemma 3.2.3]{Luothesis}.
\begin{Lemma}\label{le:charact of multipliers}
	$ \M_\Phi $ is a bounded (column) operator if and only if
	\begin{enumerate}
		\item  ${\displaystyle C_2:= \Big(\sum_{j=1}^\infty \|\phi_j\|_b^2\Big)^{\frac{1}{2}}<\infty}$, and
		\item ${\displaystyle C_3:=\sup_{z\in \D}\Big(\sum_{j=1}^\infty|\phi_j(z)|^2\Big)^{\frac{1}{2}}<\infty}$. 
	\end{enumerate}
Furthermore, 
	$
\max(C_2, C_3)\le 	\|\M_\Phi\|\le \sqrt{2}\max(C_3, C_1C_2),
$ where $C_1$ was defined in \eqref{eq:definition of C_1}.
\end{Lemma}

\begin{proof}
	Suppose that $ \M_\Phi $ is a bounded operator. Since $ 1\in \HH(b) $ and~\eqref{eq:norm in h(b)} shows that $ \|1\|_b=1 $, we conclude that 
	\[
	C_2^2=\sum_{j=1}^\infty\|\phi_j\|_b^2
	=\|\M_\Phi 1\|^2_{\bigoplus_{j=1}^\infty \HH(b)}\le \|\M_\Phi\|^2 \cdot \|1\|_b^2=\|\M_\Phi\|^2.
	\]
	This proves (a).
	To prove (b), let $N \in \N$ and $ (\gamma_j)_{j\geqslant 1}$ be a complex sequence such that $\gamma_j = 0$ when $j \geqslant N + 1$. 
	For every $ z\in \D $ and $N \in \N$, it follows from \eqref{7ygGHYGTG6} that 
	\[
	\M_\Phi^*((\gamma_j \kk^b_z)_{j \geqslant 1})=
	\sum_{j = 1}^{N}
	\M_{\phi_j}^*(\gamma_j \kk^b_z)	=	\Big(\sum_{j = 1}^{N}\gamma_j\overline{\phi_j(z)}\Big)\kk^b_z,
	\]
	and so 
	\begin{align*}
	\Big|\sum_{j = 1}^{N}\gamma_j\overline{\phi_j(z)}  \Big|\|\kk^b_z\|_{b} & \le \|	\M_{\Phi}^*\|\|(\gamma_j \kk^b_z)_{j \geqslant 1}\|_{\bigoplus_{j=1}^{N} \HH(b)}\\
	& = \|	\M_{\Phi}\|\Big(\sum_{j = 1}^{N}|\gamma_j|^2\Big)^{\frac{1}{2}}\|\kk^b_z\|_b.
	\end{align*}
	Therefore, 
	\[
		\Big|\sum_{j = 1}^{N}\gamma_j\overline{\phi_j(z)}  \Big|\le \|	\M_{\Phi}\| \Big(\sum_{j = 1}^{N}|\gamma_j|^2\Big)^{\frac{1}{2}} \quad \mbox{for all $N \in \N$}.
	\]
	Since the inequality above is true for any $ (\gamma_j)_{j\geqslant 1}$,  the Riesz representation theorem implies that 
	\[
	\sum_{j=1}^N |\phi_j(z)|^2\le \|\M_\Phi\|^2 \quad \mbox{for all $z \in \D$}.
	\]
The inequality above is true for all $N$, 
	which proves (b).
	
	Conversely, assume that conditions (a) and (b) are satisfied. For any $ f\in \HH(b) $ we have 
	\[
\|	\M_\Phi f\|^2_{\bigoplus_{j=1}^\infty \HH(b)}=\sum_{j=1}^\infty \|\phi_j f\|_{b}^2.
	\]
	From \eqref{uUUiipPPS} we can write  $ f=a_1\widetilde{f}+p $, with $ \widetilde{f}\in H^2 $ and $ p\in \P_{N-1} $. The definition of the norm on $\HH(b)$ from \eqref{eq:norm in h(b)} yields
	\begin{align*}
	\|\phi_j f\|_b^2
	& =\|a_1\widetilde{f}\phi_j +\phi_j p\|^2_b\\
	& \le 2(\|a_1\widetilde{f}\phi_j\|^2_b +\|\phi_j p\|^2_b  )\\
	&=
	2(\|\widetilde{f}\phi_j\|^2_{H^2} +\|\phi_j p\|^2_b  ).
	\end{align*}
	Hence,
	\begin{equation}\label{eq:est multiplier}
			\|\M_\Phi f\|_{\bigoplus_{j=1}^\infty \HH(b)}^2\le 
		2\Big( \sum_{j=1}^\infty \|\widetilde{f}\phi_j\|^2_{H^2}+ \sum_{j=1}^\infty \|\phi_j p\|_b^2
		\Big).
	\end{equation}

	To estimate the first sum on the right hand side of \eqref{eq:est multiplier}, we can use \eqref{knnhHarsysu} and Fubini's theorem to obtain
	\begin{align}
	 \sum_{j=1}^\infty \|\widetilde{f}\phi_j\|^2_{H^2} & =
	  \sum_{j=1}^\infty\int_\T |\phi_j(\zeta)|^2 |\widetilde{f}(\zeta)|^2\, dm(\zeta) \nonumber \\  
	  & = \int_\T  \sum_{j=1}^\infty|\phi_j(\zeta)|^2 |\widetilde{f}(\zeta)|^2\, dm(\zeta). \label{xxxoo}
	\end{align}
	From the facts that  $ \phi_j(r \zeta)\to \phi_j(\zeta) $ for almost every $\zeta \in \T$ as $r \to 1^{-}$, and
	$$\sum_{j = 1}^{N} |\phi_{j}(r \xi)|^2 \leq C_{3}^{2} \quad \mbox{for all $N \in \N$,}$$ one sees that 
	$$\sum_{j = 1}^{N} |\phi_{j}(\xi)|^2 \leq C_{3}^{2} \quad \mbox{for almost every $\xi \in \T$ and every $N \in \N$}.$$
	Now let $N \to \infty$ to conclude that
	\[
	\sum_{j=1}^\infty|\phi_j(\xi)|^2 \leq C_3^2 \quad \mbox{for almost every $\xi \in \T$}.
	\]
	Thus, continuing the estimate from \eqref{xxxoo}, we obtain 
	\[
	 \sum_{j=1}^\infty \|\widetilde{f}\phi_j\|^2_{H^2}\le C_3^2
	 \int_\T    |\widetilde{f}(\zeta)|^2\, dm(\zeta)=C_3^2\|\widetilde{f}\|_{H^2}^2.
	\]
	
	To estimate the second term on the right hand side of~\eqref{eq:est multiplier}, observe that 
	\[
	\|\phi_j p\|_b \le \|p\|_{\Mult(\HH(b))}\|\phi_j\|_b\le C_1\|p\|_{H^2}\|\phi_j\|_b,
	\]
	where $ C_1 $ is defined by~\eqref{eq:definition of C_1}.
	Therefore,
	\[
	\sum_{j=1}^\infty \|\phi_j p\|_b^2
	\le C_1^2\|p\|_{H^2}^2 \Big(	\sum_{j=1}^\infty \|\phi_j \|_b^2  \Big)=C_1^2C_2^2\|p\|_{H^2}^2.
	\]
	It follows from~\eqref{eq:est multiplier} that 
	\[
		\|\M_\Phi f\|_{\bigoplus_{j=1}^\infty \HH(b)}^2\le 2\big(C_3^2\|\widetilde{f}\|_{H^2}^2+C_1^2C_2^2\|p\|_{H^2}^2\big).
	\]
	Since $ \|f\|_b^2=\|\widetilde{f}\|_{H^2}^2+\|p\|_{H^2}^2 $ (see \eqref{eq:norm in h(b)}), we see that
	\[
		\|\M_\Phi f\|_{\bigoplus_{j=1}^\infty \HH(b)}^2\le 2\max(C_3^2, C_1^2C_2^2) \|f\|_b^2.
	\]
	Therefore, $M_{\Phi}$ is bounded and  
	$
	\|\M_\Phi\|\le \sqrt{2}\max(C_3, C_1C_2),
	$
	which finishes the proof of the lemma.
\end{proof}

Here is the  second main result of this paper, a corona theorem for $\Mult( \HH(b)) $. We remind the reader that for rational
$b\in \operatorname{ball}(H^{\infty})$ (and not a finite Blaschke product), 
there exists a Pythagorean mate $a$ to which we can associate
the polynomial $a_1(z)=\prod_{j=1}^n(z-\xi_j)^{m_j}$
as explained in (5.4).

\begin{Theorem}\label{th:corona}
	Let $ b \in \operatorname{ball}(H^{\infty}) $ be rational, but not a finite Blaschke product. Suppose that $\Phi = (\phi_j)_{j \geqslant 1}$ is a sequence in $ \Mult(\HH(b)) $ that satisfies the conditions
	\begin{itemize}
		\item[(i)] ${\displaystyle \|\M_{\Phi}\|\le 1}$, and
		\item[(ii)] ${\displaystyle 0<\delta^2\le \sum_{j=1}^\infty|\phi_j(z)|^2} $ for all  $ z\in\D $.
		
	\end{itemize}
Then there is a sequence $\B = (b_j)_{j \geqslant 1}$ in $\Mult(\HH(b)) $ such that
\begin{itemize}
	\item[(a)] ${\displaystyle  \sum_{j=1}^\infty\phi_j(z)b_j(z)=1}$ for all $ z\in \D $, and 
	
	\item[(b)] ${\displaystyle \|\M_\B\| \le \frac{C}{\delta^{\max m_j}} \Big( 1+\frac{1}{\delta^2}  \log \frac{1}{\delta} \Big)}$,
	
\end{itemize}
where $ C  > 0$ depends only on $ b $.
\end{Theorem}

\begin{proof}
	We begin by decomposing each $\phi_j \in \Mult(\HH(b))$  as 
	\begin{equation}\label{eq:decomposition phi_j}
	 \phi_{j} = a_1 \widetilde{\phi_j} + p_j, \quad \widetilde{\phi_j} \in H^2, \quad p_j \in \mathscr{P}_{N - 1},
	\end{equation}
where $ a_1 $ is the polynomial associated with the Pythagorean mate of $ b $ as defined by~\eqref{eq:definition of a}.
	Furthermore, 
	\begin{equation}\label{7665432}
	a_1 \widetilde{\phi_j} \in H^{\infty} \quad \mbox{for all $j \geqslant 1$}.
	\end{equation}
Conditions (i), (ii),  and Lemma~\ref{le:charact of multipliers} imply that
\begin{equation}\label{000oOO00oo}
\delta^2 \leq \sum_{j = 1}^\infty |\phi_j(z)|^2\le 1 \quad \mbox{for all $z\in \D$}.
\end{equation}
Apply Lemma \ref{8767ytyuhjBHHGBGHUHU} to see that for fixed $1 \leq k \leq n$ there is a $c_k > 0$, depending only on $b$, such that for every $\eta > 0$, every $z \in \D$, and every $j \geqslant 1$,
$$|\phi_j(z)|^2 \leq (1 + \eta) |\phi_j(\xi_k)|^2 + \Big(1 + \frac{1}{\eta}\Big) c_k\frac{|z - \xi_k|^2}{1 - |z|^2} \|\phi_j\|_{b}^{2}.$$ By summing over $j$ in the previous inequality, it follows from Lemma \ref{le:charact of multipliers}  that 
\begin{align*}
\delta^2 & \leq (1 + \eta) \sum_{j = 1}^{\infty} |\phi_j(\xi_k)|^2 + \Big(1 + \frac{1}{\eta}\Big) c_k \frac{|z - \xi_k|^2}{1 - |z|^2}  \sum_{j = 1}^{\infty} \|\phi_j\|_{b}^{2}\\
& \leq (1 + \eta) \sum_{j = 1}^{\infty} |\phi_j(\xi_k)|^2 + \Big(1 + \frac{1}{\eta}\Big) c_k \frac{|z - \xi_k|^2}{1 - |z|^2}.
\end{align*}
In the above, note the use of the fact that $\|\M_{\Phi}\| \leq 1$ and so $C_2 \leq 1$. 
Now let $z \to \xi_k$ radially to see that the second term above goes to zero and thus
$$\delta^2 \leq (1 + \eta) \sum_{j = 1}^{\infty} |\phi_j(\xi_k)|^2.$$
Letting $\eta \to 0^{+}$ yields 
$$\delta^2 \leq \sum_{j = 1}^{\infty} |\phi_j(\xi_k)|^2.$$
The estimate in  \eqref{000oOO00oo} and Fatou's lemma yield
$$\sum_{j  = 1}^{\infty} |\phi_j(\xi_k)|^2 \leq 1.$$ Finally, use $p_j(\xi_k) = \phi_j(\xi_k)$ to obtain
$$\delta^2 \leq \sum_{j = 1}^{\infty} |p_j(\xi_k)|^2 \leq 1.$$
	In particular, for every $ 1\le k\le n $ there exists an $\ell_k \in \N$ such that 
	$$\sum_{j = \ell_k + 1}^{\infty} |p_{j}(\xi_k)|^2 \leq \frac{\delta^2}{2}.$$
	
	Define $L = \max\{\ell_{k}: 1 \leq k \leq n\}$. Then for every $1 \leq k \leq n$, we have 
	\begin{align*}
	\sum_{j = 1}^{L} |p_{j}(\xi_k)|^2 & = \sum_{j = 1}^{\infty} |p_j(\xi_k)|^2 - \sum_{j = L + 1}^{\infty} |p_j(\xi_k)|^2\\
	& \geqslant \delta^2 - \sum_{j = \ell_k + 1}^{\infty} |p_j(\xi_k)|^2\\
	& \geqslant \delta^2 - \frac{\delta^2}{2} = \frac{\delta^2}{2} > 0.
	\end{align*}
	
	On the other hand, by assumption (i) and Lemma~\ref{le:charact of multipliers}, 
	\[ \sum_{j=1}^\infty\|\phi_j\|_b^2=C_2^2\le\|\M_{\Phi}\|^2\le 1 ,\] whence the norm \eqref{eq:norm in h(b)} implies that $ \sum_{j=1}^\infty\|p_j\|_{H^2}\le 1 $. Using the constant $ C_1 $ defined in~\eqref{eq:definition of C_1}, we see that 
	\[
\sum_{j=1}^L\|p_j\|^2
	\le C_1^2 \sum_{j=1}^L\|p_j\|_{H^2}^2
	\le C_1^2 \sum_{j=1}^\infty\|p_j\|_{H^2}^2\le C_1^2.
	\]

	 From now on, all constants will be denoted by C and may change from line to line.
	
	Now apply Theorem~\ref{th:main estimate2} with $ K=N - 1 $, $ A=a_1 $, and $ B_j=p_j $ for $1\leq j\leq L$, to produce a  $ q\in \P_{N-2} $ and $ q_1, \dots, q_L\in \P_{N-1} $ such that 
	\begin{equation}\label{eq:properties of q, q1, qL}
	\begin{split}
	&qa_1+q_1p_1+\dots+q_Lp_L \equiv 1\quad\text{and}\\
	&\Big(\|q\|_{H^2}^2 + 
	\sum_{j=1}^L\| q_j \|_{H^2}^2\Big)^{\frac{1}{2}} \le \frac{C}{\delta^{\max m_i}}
	\end{split}
	\end{equation}
	for some $C > 0$, depending only on $ a_1 $ and hence depending only on $b$. (Since all norms on $ \P_{N-1} $ are equivalent, we may replace the norm used in~\eqref{eq:estimate2} by the $H^2$-norm). We set $ q_k\equiv 0 $ for all  $ k\geqslant L+1. $
	
	By Tolokonnikov's theorem \cite{MR629839} (mentioned in the introduction), there is a universal $C > 0 $ and a sequence  $(e_j)_{j \geqslant 1}$ in $H^\infty $ such that 
	\begin{equation}\label{eq:tolokonikov}
		\begin{split}
			&\sum_{j=1}^\infty \phi_j(z) e_j(z)= 1 \text{ for all }z\in\D  , \\
			&\Big(\sup_{z\in \D}\sum_{j=1}^\infty|e_j(z)|^2\Big)^{\frac{1}{2}} \le C \frac{1}{\delta^2} \log \frac{1}{\delta}
		\end{split}
	\end{equation}
	For each $j \geqslant 1$ define 
	\begin{equation}\label{eq:definition of b_j}
			b_j :=q_j+\Big(1-\sum_{k=1}^{L}\phi_k q_k\Big)e_j.
	\end{equation}
We will now show that $(b_j)_{j \geqslant 1} $ is the required sequence satisfying conditions (a) and (b).

 First we check that $b_j \in \Mult(\HH(b))$ for all $j \geqslant 1$. 
Using~\eqref{eq:decomposition phi_j} and~\eqref{eq:properties of q, q1, qL},  one obtains
\[
\begin{split}
		\Big(1 - \sum_{k = 1}^{L} \phi_k q_k\Big) e_j&=\Big(1 - \sum_{k = 1}^{L} q_k (a_1 \widetilde{\phi_k} + p_k)\Big) e_j\\
	 &= \Big(1 - \sum_{k = 1}^{L} q_k p_k - a_1 \sum_{k = 1}^{L} q_k \widetilde{\phi_k}\Big) e_j\\
	 	& = a_1 \Big ( q  -     \sum_{k = 1}^{L} q_k \widetilde{\phi_k}\Big)e_j.
\end{split}
\]
Therefore~\eqref{eq:definition of b_j} can be written as
\begin{equation}\label{00888666Y6YY6}
	b_j=a_1 \Big ( q  -     \sum_{k = 1}^{L} q_k \widetilde{\phi_k}\Big)e_j+q_j.
\end{equation}
Since $ q_j\in\P_{N-1} $ and $q  -     \sum_{k = 1}^{L} q_k \widetilde{\phi_k} \in H^2$, this is precisely the decomposition of $ b_j $ from \eqref{uUUiipPPS}. Therefore, $b_j \in \HH(b)$ and it follows from \eqref{7665432} that $b_j \in H^{\infty}$. Thus, from \eqref{mukltiewrlierrrrrs}, $b_j \in \HH(b) \cap H^{\infty} = \Mult(\HH(b))$.

	Second, we observe that 
	\begin{align*}
		\sum_{j = 1}^{\infty} b_j \phi_j & = \sum_{j = 1}^{\infty} \Big(q_j + \big(1 - \sum_{k = 1}^{L} \phi_k q_k\big) e_j\Big) \phi_j\\
		& = \sum_{j = 1}^{\infty} \phi_j q_j + \Big(1 - \sum_{k = 1}^{L} \phi_k q_k\Big) \sum_{j = 1}^{\infty} e_j \phi_j\\
		& = \sum_{j = 1}^{\infty} \phi_j q_j + \Big(1 - \sum_{k = 1}^{L} \phi_k q_k\Big) \cdot 1 && \mbox{(by \eqref{eq:tolokonikov})}\\
		& = \sum_{j = 1}^{L} \phi_j q_j + \Big(1 - \sum_{k = 1}^{L} \phi_k q_k\Big)\\ & = 1.
	\end{align*}
	Thus (a) is proved. 

In order to prove (b) of Theorem~\ref{th:corona}, we  need to show that $ \M_\B $ satisfies inequalities (a) and (b) in Lemma~\ref{le:charact of multipliers}.  Apply \eqref{eq:norm in h(b)}  and \eqref{00888666Y6YY6} to obtain
	\begin{equation}\label{eq:intermediate}
			\sum_{j=1}^\infty\|b_j\|_b^2 
		=	\sum_{j=1}^\infty\|q_j\|_{H^2}^2
		+\sum_{j=1}^\infty \Big\|\Big ( q  -     \sum_{k = 1}^{L} q_k \widetilde{\phi_k}\Big)e_j  \Big\|_{H^2}^2.
	\end{equation}
	By \eqref{eq:properties of q, q1, qL}, the first term on the right hand side of the above is bounded by $ {\displaystyle \frac{C}{\delta^{2\max m_j}} }$. To bound  the second term, we have
		\begin{align*}
& \sum_{j=1}^\infty \Big\|\Big ( q  -     \sum_{k = 1}^{L} q_k \widetilde{\phi_k}\Big)e_j  \Big\|_{H^2}^2\\
	& =	\int_\T \sum_{j=1}^\infty \big| q(\zeta)-\sum_{k=1}^L q_k(\zeta)\widetilde{\phi_k}(\zeta)  \big|^2
		|e_j(\zeta)|^2\, dm(\zeta)\\
		&=\int_\T  \big| q(\zeta)-\sum_{k=1}^L q_k(\zeta)\widetilde{\phi_k}(\zeta)\big|^2 \big(\sum_{j=1}^\infty 
		|e_j(\zeta)|^2\big)\, dm(\zeta)\\
		&\le \Big(\frac{C}{\delta^2}\log\frac{1}{\delta}\Big)^2
			\int_\T  \big| q(\zeta)-\sum_{k=1}^L q_k(\zeta)\widetilde{\phi_k}(\zeta)  \big|^2
	 \, dm(\zeta)\qquad\text{by~\eqref{eq:tolokonikov}}\\
	 &=\Big(\frac{C}{\delta^2}\log\frac{1}{\delta}\Big)^2\Big\| q-\sum_{k=1}^L q_k\widetilde{\phi_k}\Big\|^{2}_{H^2}\\
	 & \le 2 \Big(\frac{C}{\delta^2}\log\frac{1}{\delta}\Big)^2
	 \Big(\|q\|^2_{H^2}+\Big\| \sum_{k=1}^L q_k\widetilde{\phi_k}\Big\|_{H^2}^{2}
	 \Big).
	\end{align*}
Again~\eqref{eq:properties of q, q1, qL} yields  
	${\displaystyle \|q\|^2_{H^2}\le \frac{C}{\delta^{2\max m_j}}}$.
	On the other hand, \eqref{eq:norm in h(b)} implies that
	$\|\widetilde{\phi_k}\|_{H^2}\le\|\phi_k\|_b $, which, together with the Cauchy--Schwarz inequality, yield
	\begin{align*}
			\Big\| \sum_{k=1}^L q_k\widetilde{\phi_k}\Big\|_{H^2}^2
			&\le\Big( \sum_{k=1}^L\| q_k\widetilde{\phi_k}\|_{H^2}\Big)^2 \le\Big( \sum_{k=1}^L \|q_k\|_\infty \|\phi_k\|_b\Big)^2\\
			&\le C \Big( \sum_{k=1}^L \|q_k\|_{H^2} \|\phi_k\|_b\Big)^2 \\&\le C \Big( \sum_{k=1}^L\|q_k\|_{H^2}^2 \Big)
			\Big( \sum_{k=1}^L \|\phi_k\|_b^2 \Big).
\end{align*}
	Using \eqref{eq:properties of q, q1, qL} once more, we see that the first factor in the last formula is bounded above by ${\displaystyle \frac{C}{\delta^{2\max m_j}}}$, while,   by condition (i) in the statement of the theorem and Lemma~\ref{le:charact of multipliers}, the second factor is bounded above by $1$. Consequently,
	\[
	\begin{split}
			\sum_{j=1}^\infty \Big\|\Big ( q  -     \sum_{k = 1}^{L} q_k \widetilde{\phi_k}\Big)e_j  \Big\|_{H^2}^2
&		\le \frac{C}{\delta^{2\max m_j}}\Big(\frac{1}{\delta^2} \log\frac{1}{\delta}\Big)^2.
	\end{split}
	\]
	Returning to~\eqref{eq:intermediate}, it follows that 
	\begin{equation}\label{eq:corona first estimate}
		\sum_{j=1}^\infty\|b_j\|_b^2 \le
	\frac{C}{\delta^{2\max m_j}}\Big(
	1+\frac{1}{\delta^2}\log\frac{1}{\delta}
	\Big)^2.
	\end{equation}
	proving that   the inequality (a) in Lemma~\ref{le:charact of multipliers} is satisfied.
	
	In order to prove (b), fix   $z\in \D $. From the definition of $ b_j $  from \eqref{eq:definition of b_j}, we see that 
	\[
	|b_j(z)|^2\le 2 \Big(
	|q_j(z)|^2+\big|1-\sum_{k=1}^L\phi_k(z)q_k(z)\big|^2 |e_j(z)|^2
	\Big),
	\]
	whence
	\begin{equation}\label{eq:computing b_j}
		\sum_{j=1}^\infty |b_j(z)|^2
		\le
			2\sum_{j=1}^\infty |q_j(z)|^2
			+ 2\big|1-\sum_{k=1}^L\phi_k(z)q_k(z)\big|^2
			\Big(	\sum_{j=1}^\infty |e_j(z)|^2\Big).
	\end{equation}
	From~\eqref{eq:tolokonikov} we obtain the estimate 
	\[
	\sum_{j=1}^\infty|e_j(z)|^2\le \Big(\frac{C}{\delta^2}\log \frac{1}{\delta}\Big)^2,
	\]
	while from~\eqref{eq:properties of q, q1, qL} we have the estimate 
	\[
		\sum_{j=1}^\infty|q_j(z)|^2
		=	\sum_{j=1}^L|q_j(z)|^2 \le 
		\sum_{j=1}^L\|q_j\|_\infty^2
		\le C
			\sum_{j=1}^L\|q_j\|_{H^2}^2
			\le \frac{C}{\delta^{2\max m_j}}
	\]
	(note that $ q_j \equiv0 $ for all $ j>L $).
	Finally, using condition (b) in Lemma~\ref{le:charact of multipliers} (valid for $ \M_{\Phi} $ by assumption) as well as~\eqref{eq:properties of q, q1, qL}, we see that 
	\[
	\begin{split}
\big|1-\sum_{k=1}^L\phi_k(z)q_k(z)\big|&
\le 1+\sum_{k=1}^L|\phi_k(z)q_k(z)|\\
&\le 1+ \Big(\sum_{k=1}^L|\phi_k(z)|^2   \Big)^{\frac{1}{2}}
		\Big(\sum_{k=1}^L|q_k(z)|^2   \Big)^{\frac{1}{2}}
		\\&\le 1+ \Big(\sum_{k=1}^L|\phi_k(z)|^2   \Big)^{\frac{1}{2}}
		\Big(\sum_{k=1}^L\|q_k\|_\infty^2   \Big)^{\frac{1}{2}}\\
		&\le 1+ C\Big(\sum_{k=1}^L|\phi_k(z)|^2   \Big)^{\frac{1}{2}}
		\Big(\sum_{k=1}^L\|q_k\|_{H^2}^2   \Big)^{\frac{1}{2}}\\
		&\le 1+\frac{C}{\delta^{\max m_j}}.
	\end{split}
	\]
	Gathering up all the last estimates and plugging them into~\eqref{eq:computing b_j} yields
 	\begin{equation}\label{eq:corona second estimate}
 		\sum_{j=1}^\infty |b_j(z)|^2
 	\le   \frac{C}{\delta^{2\max m_j}} \Big(1+ \frac{1}{\delta^2}\log\frac{1}{\delta} \Big)^2.
 	\end{equation}
 
 Together,~\eqref{eq:corona first estimate} and~\eqref{eq:corona second estimate} show that $ \M_\B $ satisfies conditions (a) and (b) in Lemma~\ref{le:charact of multipliers}. It is therefore a bounded operator whose norm satisfies
	\begin{equation}\label{maxxxxx}
	\|\M_\B\|\le \frac{C}{\delta^{\max m_j}} \Big(1+ \frac{1}{\delta^2}\log\frac{1}{\delta} \Big),
	\end{equation}
	for some constant $ C>0 $,
	which ends the proof of the theorem.
\end{proof}

If $\|b\|_{\infty} < 1$, it is known that $\HH(b) = H^2$, with an equivalent norm. Furthermore, in this case the Pythagorean mate $a$ will have no zeros and so the exponent on $\delta$ in \eqref{maxxxxx} will be $\max m_j= 0$. This corresponds to the estimate 
$$\|\M_\B\|  \leq C (1 + \frac{1}{\delta^2} \log \frac{1}{\delta})$$
which is the Uchiyama estimate from \eqref{uchi}. Of course, Uchiyama's result was used in our proof.

\section{Final remarks}

As noted in the introduction, Theorem~\ref{th:corona} is related to some of the results in~\cite{Luothesis, MR4363747}. Here is how one makes the connection. If $ \mu $ is a finite positive Borel measure on $ \T $ and $ P_\mu $ is its Poisson integral
$$P_{\mu}(z) := \int_{\T} \frac{1 - |z|^2}{|z - \xi|^2} d  \mu(\xi), \quad z \in \D,$$
 the {\em harmonically weighted  Dirichlet space} $\DD_\mu$, introduced by Richter \cite{richter}, is the space of  $ f\in H^2 $ satisfying
\[
\int_\D |f'(z)|^2P_\mu(z)dA(z)<\infty,
\]
where $dA$ is area measure.
In~\cite{MR3110499}, Costara and Ransford proved that a de Branges--Rovnyak space $ \HH(b) $, where $ b \in \operatorname{ball}(H^{\infty})$ is rational and not a finite Blaschke product,  coincides (with equivalent norms) with a $ \DD_\mu $  space if and only if the zeros  on $\T$ of the Pythagorean mate $ a$ of $ b$ are simple, that is, with our notation from \eqref{eq:definition of a}, when all the zeros of $ a_1 $ are simple. If this happens, then the support of $ \mu $ is precisely this set of (simple) zeros. 

For this class of  $\mathcal{D}_{\mu}$ spaces, Luo  \cite{Luothesis} proved a corona theorem, including estimates of the norm of the solutions. This turns out to be, when translated in the context of de Branges spaces,  the particular case of Theorem~\ref{th:corona} when $ m_j=1 $ for all $1 \leq j \leq n$. One sees that  our Theorem \ref{th:corona} covers the general case when the roots of $ a$ on $\T$ have arbitrary multiplicities, where $\HH(b)$ no longer coincides   with a $\mathcal{D}(\mu)$ space.  It should also be noted that part of the argument in the proof of Theorem~\ref{th:corona} is similar to an argument from~\cite[Theorem 3.2.4]{Luothesis}.

Finally,  note that in~\cite{MR4363747} Luo obtained a general corona theorem for harmonically weighted Dirichlet spaces $\mathcal{D}_{\mu}$.

\bibliographystyle{plain}

\bibliography{references}

\end{document}